\documentclass[12pt,a4paper]{amsart}
\usepackage{amsfonts}
\usepackage{amsthm}
\usepackage{amsmath,amssymb,amsfonts,mathtools,enumerate,tikz-cd}
\usepackage{amscd}
\usepackage[latin2]{inputenc}
\usepackage{t1enc}
\usepackage[mathscr]{eucal}
\usepackage{indentfirst}
\usepackage{graphicx}
\usepackage{graphics}
\usepackage{pict2e}
\usepackage{epic}
\numberwithin{equation}{section}
\usepackage[margin=2.9cm]{geometry}
\usepackage{epstopdf} 
\usepackage{hyperref}
\usepackage{multirow}

\theoremstyle{plain}
\newtheorem{thm}{Theorem}[section]

\newtheorem{prop}[thm]{Proposition}

 \theoremstyle{definition}
\newtheorem{defn}[thm]{Definition}
\newtheorem{conj}[thm]{Conjecture}

\newcommand{\Hom}{{\rm{Hom}}}

\newcommand{\cl}[1]{\mathcal{#1}}
\newcommand{\bb}[1]{\mathbb{#1}}
\newcommand{\tr}{\operatorname{tr}}
\newcommand{\PSL}{\operatorname{PSL}}
\newcommand{\PGL}{\operatorname{PGL}}

\newcommand{\Gal}{\operatorname{Gal}}
\newcommand{\hfal}{\operatorname{h_{Fal}}}
\newcommand{\Ind}{\operatorname{Ind}}
\newcommand{\height}{\operatorname{ht}}

\begin{document}

\title{Signature $(n-2,2)$ CM Types and the Unitary Colmez Conjecture}

\author[S. Parenti]{Solly Parenti}

\address{University of Wisconsin-Madison \\ Department of Mathematics \\
Madison WI 53703} 

\email{sparenti@wisc.edu}



\begin{abstract} Colmez conjectured a formula relating the Faltings height of CM abelian varieties to a certain linear combination of log derivatives of $L$-functions.  In this paper, we study the case of unitary CM fields and by studying the class functions that arise, we reduce the conjecture to a special case. Using the Galois action, we prove more cases of the Colmez Conjecture.
\end{abstract}

\maketitle

\section{Set Up and Theorem}
%We give a brief introduction to the Colmez conjecture and refer the reader to a previous paper of the author \cite{ParentiPSL2} for a more thorough introduction.

%Given a CM field $E$ and a CM type $\Phi$ of $E$, let $\Phi^c$ denote the extension of $\Phi$ to $E^c$, the Galois closure of $E$ and consider $\Phi^c$ as a function $\Phi^c:\Gal(E^c/\bb{Q}) \rightarrow \{0,1\}$. Let $\widetilde{\Phi^c}$ denote the reflex CM type and let $A_{\Phi}$ denote the convolution of $\Phi^c$ and $\widetilde{\Phi^c}$.  Denote by $A_{\Phi}^0$ the projection of $A_{\Phi}$ onto the space of class functions on $\Gal(E^c/\bb{Q})$

%Let $F$ be a totally real number field of degree $n$ over $\bb{Q}$, let $F^c$ be the Galois closure of $F$ and let $G=\text{Gal}(F^c/\bb{Q})$.  For $k$ an imaginary quadratic field, let $E$ be the CM field $E:= kF$ and let $E^c$ be its Galois closure which is given by $E^c = kF^c$.  Let $\rho:E \rightarrow E$ be the complex conjugation on $E$.  Let $H = \text{Gal}(F^c/F)$, where $\#H = h$, and write $\{\sigma_1,\ldots,\sigma_n\}$ for coset representatives of $G/H$, which we identify with embeddings $E \hookrightarrow \bb{C}$.  

The Colmez conjecture gives a formula for the Faltings height of a CM abelian variety in terms of log derivatives of $L$-functions arising from the CM type.  This conjecture has proven useful in giving bounds for the Faltings height of CM abelian varieties (see \cite{ColSurLa} for the case of elliptic curves and \cite{TsimermanAndreOort} where a weaker form of the Colmez conjecture is used in the proof of the Andr\'e-Oort conjecture for the moduli space of principally polarized abelian varieties).

\begin{defn} A unitary CM field $E$ is a CM field of the form $E=kF$, where $F$ is a totally real number field and $k \subseteq \bb{C}$ is an imaginary quadratic field.
\end{defn}

Given a CM field $E$ of degree $2n$, a CM type of $E$ consists of $n$ embeddings of $E$ into $\bb{C}$ such that no two of the embeddings differ by complex conjugation.  For unitary CM fields, we will stratify the CM types by signature.

\begin{defn}
Let $E=kF$ be a unitary CM field.  A CM type $\Phi \subseteq \Hom(E,\bb{C})$ has signature $(n-r,r)$ if exactly $n-r$ of the embeddings in $\Phi$ restrict to the identity $k \hookrightarrow \bb{C}$.  
\end{defn}

The main theorem of this paper is that we can reduce the Colmez conjecture in the unitary case to CM types of signature $(n-2,2)$.  

\begin{thm}
Let $E=kF$ be a unitary CM field.  Then, the Colmez conjecture holds for $E$ if and only if it holds for all CM types of signature $(n-2,2)$.  
\end{thm}

We spend the remainder of this section providing a brief description of the Colmez conjecture and refer the reader to \cite{YangYin} giving a more thorough background to the conjecture.  Section 2 contains the proof of Theorem 1.3 and in Section 3 we apply Theorem 1.3 to obtain examples of CM fields where the Colmez conjecture holds.

%Given a CM type $\Phi$ of $E$, we compute the corresponding class function $A_{\Phi}^0$.  We then rewrite $A_{\Phi}^0$ in terms of CM types of signatures $(n,0), (n-1,1),$ and $(n-2,2)$.  Then, the linearity of Colmez's original constructions together with recent refinements of the average form of the Colmez conjecture will imply the result.

%In the remainder of this section, we offer a brief description of the Colmez conjecture and refer the reader to a previous paper \cite{YangYin} providing a more thorough background to the conjecture. Section 2 contains a proof of Theorem 1.1 and Section 3 contains further comments relating to other work.

Let $\Phi$ be a CM type of a CM field $E$ and identify $\Phi$ with its characteristic function $\Phi:\Hom(E,\bb{C}) \rightarrow\{0,1\}$. If $E^c$ denotes the Galois closure of $E$ (which is also a CM field), then the restriction map $\Hom(E^c,\bb{C}) \rightarrow \Hom(E,\bb{C})$ can be used to extend $\Phi$ to $\Phi^c$, a CM type on $E^c$.

Choosing an identification of $\Hom(E^c,\bb{C})$ with $\Gal(E^c/\bb{Q})$, we obtain a function $\Phi^c:\Gal(E^c/\bb{Q}) \rightarrow \{0,1\}$ and we define the reflex CM type $\widetilde{\Phi^c}:\Gal(E^c/\bb{Q}) \rightarrow\{0,1\}$ by $\widetilde{\Phi^c}(g) = \Phi^c(g^{-1})$.  

Let $A_{\Phi}:\Gal(E^c/\bb{Q}) \rightarrow \bb{C}$ denote the function we obtain by taking a normalized convolution of $\Phi^c$ and $\widetilde{\Phi^c}$.  More concretely,
\[
A_{\Phi}(g) = \frac{1}{\#\Gal(E^c/\bb{Q})}\sum_{\sigma \in \Gal(E^c/\bb{Q})}\Phi^c(\sigma) \widetilde{\Phi^c}(\sigma^{-1}g).
\]
%%% Maybe include this stuff only in section 2, when we actually use it to compute things. Also, make sure we don't need a factor of 1/size of group on bottom.  Woops, that's all from Tonghai.
%Let $\operatorname{Maps}(\Gal(E^c/\bb{Q}),\bb{C})$ denote the set of all maps from $\Gal(E^c/\bb{Q}) 
%We will often use the isomorphism $\operatorname{Maps}(\Gal(E^c/\bb{Q}),\bb{C}) \cong \bb{C}[\Gal(E^c/\bb{Q})]$ between the ring of all maps from $\Gal(E^c/\bb{Q})$ to $\bb{C}$ (a ring under convolution) and the group algebra of $\Gal(E^c/\bb{Q})$.  
%An isomorphism we will often use is the map $\operatorname{Maps}(\Gal(E^c/\bb{Q}),\bb{C}) \cong \bb{C}[\Gal(E^c/\bb{Q})]$.
To obtain a class function $A_{\Phi}^0$, we take the average of $A_{\Phi}$ among conjugates in $\Gal(E^c/\bb{Q})$.  That is to say,
\[
A_{\Phi}^0(g) = \frac{1}{\#\Gal(E^c/\bb{Q})}\sum_{h \in \Gal(E^c/\bb{Q})} A_{\Phi}(hgh^{-1}).
\]

As $A_{\Phi}^0$ is a class function, we may write 
\[
A_{\Phi}^0= \sum_{\chi} a_{\chi}\chi,
\]
where $a_{\chi} \in \bb{C}$ and $\chi$ ranges through the irreducible representations of $\Gal(E^c/\bb{Q})$. Then define the function $Z(s,A_{\Phi}^0)$ by
\[
Z(s,A_{\Phi}^0):= \sum_{\chi}a_{\chi}Z(s,\chi), \quad \quad Z(s,\chi):= \frac{L'(s,\chi)}{L(s,\chi)} + \frac{1}{2}\log f_{\chi},
\]
where $L(s,\chi)$ is the Artin $L$-function of $\chi$ and $f_{\chi}$ is the Artin conductor of $\chi$.

Let $\bb{Q}^{\text{cm}}$ denote the compositum of all CM number fields.  This field is an infinite degree Galois extension of $\bb{Q}$ with a well defined complex conjugation, which we will denote by $\rho$.  Via the quotient map $\Gal(\bb{Q}^{\text{cm}}/\bb{Q}) \twoheadrightarrow\Gal(E^c/\bb{Q})$, we may consider $A_{\Phi}^0$ as a class function on $\Gal(\bb{Q}^{\text{cm}}/\bb{Q})$.  

The other side of the conjecture involves the Faltings height of CM abelian varieties.  If $\Phi$ is a CM type of a CM field $E$, let $X_{\Phi}$ be an abelian variety of dimension $n$ with CM by $(\cl{O}_E,\Phi)$.  We can find a number field $L$ over which $X_{\Phi}$ has everywhere good reduction.  Denote by $\cl{X}_{\Phi}$ the Neron model of $X_{\Phi}$ defined over $\cl{O}_L$ and let $\epsilon:\operatorname{Spec}(\cl{O}_L) \rightarrow \cl{X}_{\Phi}$ be the zero section.  Define $\omega_{\cl{X}_{\Phi}}$ as the following line bundle on $\operatorname{Spec}(\cl{O}_L)$,
\[
\omega_{\cl{X}_{\Phi}}=\epsilon^{\ast}(\Omega^n_{\cl{X}_{\Phi}/\cl{O}_L}),
\]
and for a non-zero section $\alpha$ of $ \omega_{\cl{X}_{\Phi}}$, we define the (stable) Faltings height of $X_{\Phi}$ via 
\[
\hfal(X_{\Phi}):= \frac{-1}{2[L:\bb{Q}]} \sum_{\sigma:L \hookrightarrow \bb{C}} \log \Bigg\lvert \int_{X_{\Phi}^{\sigma}(\bb{C})} \alpha^{\sigma} \wedge \overline{\alpha^{\sigma}}\Bigg\rvert + \log \lvert \omega_{\cl{X}_{\Phi}/\cl{O}_L} / \cl{O}_L \alpha \rvert.
\]

This definition of stable Faltings height is independent of choice of $\alpha$ and choice of $L$ over which $X_{\Phi}$ obtains everywhere good reduction.

In his 1993 paper \cite{Col}, Colmez looks at $\cl{CM}^0$, the $\bb{Q}$ vector space of class functions $f:\Gal(\bb{Q}^{\text{cm}}/\bb{Q}) \rightarrow \bb{Q}$ such that $f(g) + f(\rho g)$ is independent of $g \in \Gal(\bb{Q}^{\text{cm}}/\bb{Q})$. One can check that for every CM type $\Phi$, the function $A_{\Phi}^0$ is an element of $\cl{CM}^0$. Colmez defines a $\bb{Q}$-linear height function $\height:\cl{CM}^0 \rightarrow \bb{R}$ such that if $X_{\Phi}$ is an abelian variety with CM by $(\cl{O}_E,\Phi)$, then 
\[
\hfal(X_{\Phi}) = -\height(A_{\Phi}^0).
\]
Here, $E$ is a CM field, $\Phi$ is a CM type of $E$, and $\hfal$ denotes the Faltings height of an abelian variety.  The $\bb{Q}$-linearity of Colmez's $\height$ will be important to us later.  Colmez's conjecture is the following alternate formula for $\height(A_{\Phi}^0)$.

\begin{conj}[Colmez]
For any CM type $\Phi$, $\height(A_{\Phi}^0) = Z(0,A_{\Phi}^0)$.
\end{conj}

\section{Proof of Theorem}
Before we start the proof of Theorem 1.3, let us introduce some notation.  If $F$ is a totally real number field of degree $n$, denote by $F^c$ the Galois closure of $F$.  Let $k$ be an imaginary quadratic field and let $E:=kF$ be a unitary CM field with complex conjugation $\rho$. We will denote the Galois closure of $E$ by $E^c$, and thus $E^c = kF^c$.  Furthermore, let $H := \Gal(F^c/F) \leq \Gal(F^c/\bb{Q})=:G$ and suppose $\#H = h$.  Then, we can identify the embeddings of $F$ into $\bb{C}$, which we will call $\{\sigma_1,\ldots,\sigma_n\}$, with coset representatives for $H \backslash G$.

An embedding $E \hookrightarrow \bb{C}$ is uniquely determined by a pair of embeddings $F \hookrightarrow \bb{C}$ and $k \hookrightarrow\bb{C}$.  We denote by $\{1,\rho\}$ the two embeddings of $k$ into $\bb{C}$ and for an embedding $\sigma:F \rightarrow \bb{C}$, we write $\rho^{i}\sigma$ for the embedding of $E$ into $\bb{C}$ given by the pair $\{\rho^i,\sigma\}$.  If $i=0$, we simply write $\sigma$ for $1\sigma$.

A CM type of $E:=kF$ consists of a choice of one of the embeddings $k \hookrightarrow \bb{C}$ for each embedding of $F \hookrightarrow \bb{C}$. Thus we can parametrize CM types of $E$ via subsets of $\{1,2,\ldots,n\}$. Given $S \subseteq \{1,2,\ldots,n\}$, the corresponding CM type of $E$ is given by
\[
\Phi_S= \{\rho^{j_i}\sigma_i : j_i = 1 \text{ if } i \in S, j_i = 0 \text{ if } i \not\in S\}.
\]
Then, $\Phi_S$ is a CM type of signature $(n-\epsilon,\epsilon)$, where $\epsilon = \#S$. We will often write CM types as sums, 
\begin{align*}
\Phi_S&= \sum_{i \in S}\rho\sigma_i + \sum_{i \not\in S} \sigma_i \\
&= \tr_{E/k} + (\rho-1)\sum_{i \in S} \sigma_i.
\end{align*}

The first step in the proof of the theorem is an explicit calculation of $A_{\Phi_S}^0$.
\begin{thm}
Let $S \subseteq\{1,2,\ldots n\}$ be of size $\epsilon$.  Then,
\begin{align*}
A_{\Phi_S}^0 &= \frac{1}{2}\tr_{E^c/k} - \frac{\epsilon}{n}(1-\rho)\tr_{E^c/k} + \frac{\epsilon}{n^2}(1-\rho)\chi_{\text{Ind}_H^G(\chi_0)}\\
&+ \frac{1}{hn^2}(1-\rho)\sum_{g \in G}g \Bigg( \sum_{i \neq j \in S} \sigma_iH\sigma_j^{-1} \Bigg) g^{-1}.
\end{align*}
\end{thm}
\begin{proof}
Recall that \[
\Phi_S = \tr_{E/k}+(\rho-1)\sum_{i \in S} \sigma_i.
\]
Extending $\Phi_S$ to $\Phi_S^c$, the CM type on $E^c$, amounts to determining which embeddings $E^c \hookrightarrow \bb{C}$ when restricted to $E$ are in $\Phi_S$. Since $E$ is the fixed field in $E^c$ by the subgroup $H$, $\Phi_S^c$ is given by
\[
\Phi_S^c = \tr_{E^c/k} + (\rho-1)\sum_{i \in S}  \sigma_i H.
\]
When we write $\Phi_S^c$ as a sum in this manner, we are interpreting $\Phi_S^c$ as an element of $\bb{C}[\Gal(E^c/\bb{Q})]$ which is isomorphic (as a ring) to the ring (under convolution) of all maps from $\Gal(E^c/\bb{Q})$ to $\bb{C}$. Next we find the reflex type $\widetilde{\Phi_S^c}$ by inverting every element in $\Phi_S^c$,
\[
\widetilde{\Phi_S^c} = \tr_{E^c/k} + (\rho-1)\sum_{j\in S} H\sigma_j^{-1}.
\]
Then take the convolution of $\Phi_S^c$ and $\widetilde{\Phi_S^c}$,
\begin{align*}
A_{\Phi_S} &= \frac{1}{[E^c:\bb{Q}]} \Phi^c\widetilde{\Phi^c}\\
&= \frac{1}{2hn}\Bigg(\tr_{E^c/k} + (\rho-1)\sum_{i\in S}  \sigma_i H\Bigg)\Bigg( \tr_{E^c/k} + (\rho-1)\sum_{j\in S} H\sigma_j^{-1}\Bigg) \\
&= \frac{1}{2}\tr_{E^c/k} - \frac{\epsilon}{n}(1-\rho)\tr_{E^c/k} + \frac{1}{n}(1-\rho) \sum_{i,j\in S}  \sigma_iH\sigma_{j}^{-1}.
\end{align*}
Finally, we need to project $A_{\Phi_S}$ onto the space of class functions to obtain $A_{\Phi_S}^0$,
\begin{equation}
A_{\Phi_S}^0 = \frac{1}{2}\tr_{E^c/k} - \frac{\epsilon}{n}(1-\rho)\tr_{E^c/k} + \frac{1}{n}(1-\rho) \frac{1}{hn}\sum_{g \in G} g \Bigg( \sum_{i,j\in S}\sigma_iH\sigma_j^{-1}\Bigg) g^{-1}.
\end{equation}

The main difficulty in (2.1) is the final term.  We first look at the elements of the sum with $i=j$.
%\[
%\star_S = \frac{1}{hn}\sum_{g \in G} g \Bigg( \sum_{i\in S}\sigma_iH\sigma_i^{-1}\Bigg) g^{-1} + \frac{1}{hn}\sum_{g \in G} g \Bigg( \sum_{i\neq j \in S}\sigma_iH\sigma_j^{-1}\Bigg) g^{-1}.
%\]
%If we first look at the terms of $\star_S$ where $i = j$, we can make the following simplification,
\begin{align}
\frac{1}{hn}\sum_{g \in G} g \Bigg( \sum_{i\in S}\sigma_iH\sigma_i^{-1}\Bigg) g^{-1}&= \frac{1}{hn}\sum_{i \in S}\Bigg( \sum_{g \in G} g \sigma_i H \sigma_i^{-1}g^{-1} \Bigg) \\
&= \frac{1}{hn} \sum_{i\in S}  \Bigg(\sum_{g \in G} gHg^{-1} \Bigg) \\
&= \frac{\epsilon}{hn}\sum_{g \in G} gHg^{-1}.
\end{align}
The following proposition simplifies (2.4) and combining the following Proposition with equation (2.1) concludes the proof.
\end{proof}
\begin{prop}
Let $\chi_0:H \rightarrow\bb{C}$ be the trivial character.  As functions $G \rightarrow \bb{C}$, we have the relation
\[
\sum_{g \in G} gHg^{-1} = h\chi_{\Ind_H^G(\chi_0)}.
\]
\end{prop}
\begin{proof}
This is proven on page 18 of \cite{YangYin}, but we sketch a proof.  Recall that the representation $\Ind_H^G(\chi_0)$ is given by 
\[
\Ind_H^G(\chi_0) = \{ f:G \rightarrow \bb{C} : f(x g) = f(g)\quad \forall x \in H, g \in G\},
\]
where $G$ acts by right translation.  The space $\Ind_H^G(\chi_0)$ consists exactly of the functions $f:H\backslash G \rightarrow \bb{C}$.  Therefore a standard calculation shows that the representation $\Ind_H^G(\chi_0)$ is isomorphic to the representation arising from the action of $G$ on $H \backslash G$ via $g \cdot H\sigma := H \sigma g^{-1}$.  Recall that we have identified $\{\sigma_1,\ldots,\sigma_n\}$ with coset representatives for $H \backslash G$.  

It is straightforward to compute the character of a permutation representation, namely it is the number of fixed points.  That is to say, for $\sigma \in G$, 
\begin{align*}
\chi_{\Ind_H^G(\chi_0)}(\sigma) &= \#\{i\in \{1,\ldots,n\} : \sigma \cdot H \sigma_i = H \sigma_i \}\\
&= \#\{i \in \{1,\ldots,n\} : \sigma \in \sigma_i H \sigma_i^{-1}\}.
\end{align*}
On the other hand,
\[
\Bigg(\sum_{g \in G} gHg^{-1} \Bigg)(\sigma) = \# \{ g \in G : \sigma \in gHg^{-1}\}.
\]
However, for a given $i$ with $\sigma \in \sigma_i H \sigma_i^{-1}$, then every $g \in \sigma_i H$ satisfies $\sigma \in gHg^{-1}$ and since $\# H = h$, we obtain $\displaystyle \sum_{g \in G} gHg^{-1} = h \chi_{\Ind_H^G(\chi_0)}$.

%has a basis $\{f_1,\ldots, f_n\}$ where
 %\begin{displaymath}
   %f_i(g) = \left\{
     %\begin{array}{lr}
       %1 & : g \in H\sigma_i\\
       %0 & : g \not\in H\sigma_i
     %\end{array}
   %\right.
%\end{displaymath} 
\end{proof}

%Our calculation shows that for any subset $S \subseteq \{1,2,\ldots,n\}$ of size $\epsilon$, we have
%\begin{align*}
%A_{\Phi_S}^0 &= \frac{1}{2}\tr_{E^c/k} - \frac{\epsilon}{n}(1-\rho)\tr_{E^c/k} + \frac{\epsilon}{n^2}(1-\rho)\chi_{\text{Ind}_H^G(\chi_0)}\\
%&+ \frac{1}{hn^2}(1-\rho)\sum_{g \in G}g \Bigg( \sum_{i \neq j \in S} \sigma_iH\sigma_j^{-1} \Bigg) g^{-1}.
%\end{align*}

Let us record a few particular cases of Theorem 2.1 which will be of use: 
\[
A_{\Phi_{\{\varnothing\}}}^0 = \frac{1}{2}\tr_{E^c/k},
\]
\[
A_{\Phi_{\{i\}}}^0= \frac{1}{2}\tr_{E^c/k} - \frac{1}{n}(1-\rho)\tr_{E^c/k} + \frac{1}{n^2}(1-\rho)\chi_{\text{Ind}_H^G(\chi_0)},
\]
\begin{align*}
A_{\Phi_{\{i,j\}}}^0 &= \frac{1}{2}\tr_{E^c/k} -\frac{2}{n}(1-\rho)\tr_{E^c/k} + \frac{2}{n^2}(1-\rho)\chi_{\text{Ind}_H^G(\chi_0)} \\
&+\frac{1}{hn^2}(1-\rho) \sum_{g \in G} \Big(g \sigma_iH \sigma_j^{-1} g^{-1} + g\sigma_j H \sigma_i^{-1}g^{-1}\Big).
\end{align*}

\begin{prop}
For any subset $S \subseteq \{1,2,\ldots,n\}$ of size $\epsilon$, we have
\[
A_{\Phi_S}^0 = \sum_{\{i,j\} \subseteq S} A_{\Phi_{\{i,j\}}}^0 - (\epsilon-2)\sum_{i \in S}A_{\Phi_{\{i\}}}^0 + \frac{(\epsilon-1)(\epsilon-2)}{2}A_{\Phi_{\{ \varnothing\}}}^0.
\]
\end{prop}
\begin{proof}
This proposition is clear when $\epsilon =0,1$, where we interpret an empty sum as 0. For $\epsilon\geq 2$, we have
\begin{align*}
\sum_{\{i,j\} \subseteq S}A_{\Phi_{\{i,j\}}}^0 &= \sum_{\{i,j\} \subseteq S}\Bigg( \frac{1}{2}\tr_{E^c/k} -\frac{2}{n}(1-\rho)\tr_{E^c/k} + \frac{2}{n^2}(1-\rho)\chi_{\text{Ind}_H^G(\chi_0)} \\
&+\frac{1}{hn^2}(1-\rho)\sum_{g \in G} \Big( g \sigma_i H \sigma_j^{-1} g^{-1} + g \sigma_jH\sigma_i^{-1}g^{-1}\Big) \Bigg)\\
&= \frac{\epsilon(\epsilon-1)}{4}\tr_{E^c/k} - \frac{\epsilon(\epsilon-1)}{n}(1-\rho)\tr_{E^c/k} + \frac{\epsilon(\epsilon-1)}{n^2}(1-\rho)\chi_{\text{Ind}_H^G(\chi_0)} \\
&+ \frac{1}{hn^2}(1-\rho) \sum_{g \in G} g\Big( \sum_{i\neq j \in S}\sigma_iH\sigma_j^{-1}\Big) g^{-1}.
\end{align*}

From Theorem 2.1, we can conclude that 
\begin{align*}
A_{\Phi_S}^0 - \sum_{\{i,j\} \subseteq S}A_{\Phi_{\{i,j\}}}^0 &= \frac{-(\epsilon+1)(\epsilon-2)}{4}\tr_{E^c/k} + \frac{\epsilon(\epsilon-2)}{n}(1-\rho)\tr_{E^c/k}\\ 
&-\frac{\epsilon(\epsilon-2)}{n^2}(1-\rho)\chi_{\text{Ind}_H^G(\chi_0)}\\
&= -(\epsilon-2)\sum_{i \in S}A_{\Phi_{\{i\}}}^0 + \frac{(\epsilon-1)(\epsilon-2)}{2} A_{\Phi_{\{\varnothing\}}}^0.
\end{align*}
\end{proof}

For a given CM type $\Phi_S$, we have written $A_{\Phi_S}^0$ in terms of the corresponding function $A_{\Phi}^0$ for CM types of signatures $(n,0), (n-1,1),$ and $(n-2,2)$.  To finish off the proof, we will use Colmez's $\height$ function as well as recent work of Yang and Yin \cite{YangYin} proving the Colmez conjecture for CM types of signature $(n,0)$ and $(n-1,1)$.

%For a given CM type $\Phi_S$, we have written $A_{\Phi_S}^0$ in terms of the corresponding function $A_{\Phi}^0$ for CM types of signatures $(n,0), (n-1,1),$ and $(n-2,2)$.  We will use recent work of Yang and Yin \cite{YangYin}

%We also use recent work of Yang and Yin \cite{YangYin} building off of the proof of the Average Colmez conjecture due to proving the Colmez conjecture for CM types of signature $(n,0)$ and $(n-1,1)$.
\begin{thm}
Suppose the Colmez conjecture holds for all CM types of $E$ of signature $(n-2,2)$. Then the Colmez conjecture holds for all CM types of $E$.
\end{thm}
\begin{proof}
Recall that Colmez defined a $\bb{Q}$-linear height function $\height:\cl{CM}^0 \rightarrow \bb{R}$ and  conjectured that $\height(A_{\Phi}^0) = Z(0,A_{\Phi}^0)$ for any CM type $\Phi$ of a CM field.  We note that $Z(0,\cdot)$ is also a $\bb{Q}$-linear function on $\cl{CM}^0$.

%The reason that $Z(0,A_{\Phi}^0) \in \bb{R}$  (rather than just $\in \bb{C}$) is Lemme 2 in Sur la hauteur.

Let $S \subseteq \{1,2,\ldots,n\}$ be a subset of size $\epsilon$.  In \cite{YangYin}, Yang and Yin show that the Colmez conjecture holds for all CM types of signature $(n,0)$ and $(n-1,1)$.  Then, supposing that $\height(A_{\Phi_{\{i,j\}}}^0) = Z(0,A_{\Phi_{\{i,j\}}}^0)$ for any $i,j$, we obtain
\begin{align*}
\height(A_{\Phi_S}^0) &= \height\Bigg(\sum_{\{i,j\} \subseteq S}A_{\Phi_{\{i,j\}}}^0-(\epsilon-2)\sum_{i \in S}A_{\Phi_{\{i\}}}^0 + \frac{(\epsilon-1)(\epsilon-2)}{2}A_{\Phi_{\{\varnothing\}}}^0\Bigg)\\
&=\sum_{\{i,j\} \subseteq S}\height(A_{\Phi_{\{i,j\}}}^0) - (\epsilon-2)\sum_{i \in S}\height(A_{\Phi_{\{i\}}}^0)+\frac{(\epsilon-1)(\epsilon-2)}{2}\height(A_{\Phi_{\{\varnothing\}}}^0) \\
&= \sum_{\{i,j\} \subseteq S}Z(0,A_{\Phi_{\{i,j\}}}^0) - (\epsilon-2)\sum_{i \in S}Z(0,A_{\Phi_{\{i\}}}^0) + \frac{(\epsilon-1)(\epsilon-2)}{2}Z(0,A_{\Phi_{\{\varnothing\}}}^0)\\
&= Z\Bigg(0, \sum_{\{i,j\} \subseteq S}A_{\Phi_{\{i,j\}}}^0-(\epsilon-2)\sum_{i \in S}A_{\Phi_{\{i\}}}^0+ \frac{(\epsilon-1)(\epsilon-2)}{2}A_{\Phi_{\{\varnothing\}}}^0\Bigg) \\
&= Z(0,A_{\Phi_S}^0)
\end{align*}

\end{proof}
\section{Galois Action on CM Types}
There is an action of Galois on the set of CM types. Namely, $g\in \Gal(E^c/\bb{Q})$ acts on a CM type $\Phi$ by $g\cdot \Phi := \{ g\sigma : \sigma \in \Phi\}$.  It is well known and straightforward to check that if $\Phi_1$ and $\Phi_2$ are two CM types that are equivalent under this action, then $A_{\Phi_1}^0 = A_{\Phi_2}^0$.

As in Section 2, let $F$ be a totally real number field with Galois closure $F^c$, with $G:=\Gal(F^c/\bb{Q})$ and $H:=\Gal(F^c/F)$.  Let $k$ be an imaginary quadratic field and consider the unitary CM field $E:=kF$.  We can describe the action of $\Gal(E^c/\bb{Q}) \cong G \times \bb{Z}/2\bb{Z}$ on the set of CM types.  The $\bb{Z}/2\bb{Z}$ component acts as complex conjugation, taking a CM type of signature $(n-\epsilon,\epsilon)$ to a CM type of signature $(\epsilon, n-\epsilon)$.  The action of $\Gal(E^c/k)$ fixes the signature of a CM type and this action on CM types of signature $(n-\epsilon,\epsilon)$ is isomorphic to the action of $G$ on the set of subsets of $G/H$ of size $\epsilon$.

%The action of $\Gal(E^c/k) \cong G $ on the set of CM types of signature $(n-\epsilon,\epsilon)$ is isomorphic to the action of $G$ on the set of subsets of $G/H$ of size $\epsilon$, and the $\Gal(k/\bb{Q})$ action corresponds to complex conjugation taking 

%In the setup of Section 2 of this paper, since $G$ acts transitively on $G/H$, all CM types of signature $(n-1,1)$ are equivalent.  This explains the fact that for any $i \in \{1,2,\ldots,n\}$, the function $A_{\Phi_{\{i\}}}^0$ is independent of $i$.

One of the main results of \cite{YangYin} is that the Colmez conjecture holds if we average amongst CM types of a given signature.  That is to say, if $\Phi(E)_{\epsilon}$ denotes all CM types of $E$ of signature $(n-\epsilon,\epsilon)$, then 
\[
\sum_{\Phi \in \Phi(E)_{\epsilon}}\height(A_{\Phi}^0) = \sum_{\Phi \in \Phi(E)_{\epsilon}} Z(0,A_{\Phi}^0).
\]
If there is only one equivalence class of CM types in a given signature, then their result immediately implies that the Colmez conjecture holds for the CM type of that signature.  This is the idea behind Yang and Yin's proof that the Colmez conjecture holds for CM types of signature $(n,0)$ and $(n-1,1)$.  In particular, combining these ideas with our Theorem 2.4 gives the following theorem.
\begin{thm}
Let $k$ be an imaginary quadratic field and let $F$ be a totally real number field with Galois closure $F^c$.  Let $H:=\Gal(F^c/F) \leq \Gal(F^c/\bb{Q}) =: G$.  If $G$ acts 2-transitively on $G/H$, then the Colmez conjecture holds for every CM type of the unitary CM field $E:=kF$.
\end{thm}

%The doubly transitive action is exactly the condition that guarantees that all the CM types of signature $(n-2,2)$ are equivalent. This theorem explains an observation in the paper \cite{ParentiPSL2}. In that paper, we considered the case where $G = \PSL_2(\bb{F}_q)$ and $H=B$, the subgroup of upper triangular matrices.  Lemma 4.2 of that paper showed that all CM types of signature $(n-2,2)$ are equivalent and thus the main result of that paper immediately follows from Theorem 3.1.

We note that there are examples of such a $(G,H)$.  In particular, $\PSL_n(\bb{F}_q)$ and $\PGL_n(\bb{F}_q)$ both act 2-transitively on $\bb{P}^{n-1}(\bb{F}_q)$, and so we can take $H$ to be the stabilizer of a point.  More precisely, we may take 
\[
H:= \begin{bmatrix} a_{11}& a_{12} & \ldots  & a_{1n} \\ 0 & a_{22}& \ldots & a_{2n} \\ \vdots & \ddots\\ 0 & a_{n2}& \ldots & a_{nn}  \end{bmatrix}
\]
In the case that $n=2$, $H$ is the Borel subgroup of upper triangular matrices and we recover the result from \cite{ParentiPSL2}.

For small values of $n,q$, the groups $\PSL_n(\bb{F}_q)$ and $\PGL_n(\bb{F}_q)$ can be realized as the Galois groups of totally real number fields over $\bb{Q}$.  In particular, consulting the LMFDB \cite{lmfdb} shows that for $q \leq 11$, the groups $\PSL_2(\bb{F}_q)$ and $\PGL_2(\bb{F}_q)$ appear as the Galois groups of totally real fields.

Furthermore, the fact that $G$ is the Galois group of the Galois closure of $F$, and $H$ is the subgroup that fixes the field $F$ implies that the action of $G$ on $G/H$ induces an embedding $G \hookrightarrow\operatorname{Sym}(G/H)$ of $G$ into the symmetric group on the set of cosets of $G$ by $H$.

Thus, we may apply Theorem 3.1 to any $G$ which is a doubly transitive subgroup of a symmetric group.  As a corollary of the classification of finite simple groups, a classification of doubly transitive groups is known.  There are infinite families of examples and sporadic examples. We list the other doubly transitive groups and refer the reader to \cite{Diximer} and \cite{BHRD} for further details on  these groups and their doubly transitive actions.

The alternating and symmetric groups $A_n$ and $S_n$ are doubly transitive subgroups of $S_n$ for any $n$.  For these groups, we take $H$ to be $A_{n-1}$ and $S_{n-1}$ respectively.  The Colmez conjecture was already discussed in this case by Yang and Yin \cite{YangYin}.

Another family of examples is the symplectic groups $\operatorname{Sp}_{2m}(\bb{F}_2)$ for $m \geq 2$.  For these groups, two choices of subgroups induce doubly transitive actions. This discussion involves quadratic forms, symmetric bilinear forms, and alternating bilinear forms in characteristic 2, so we will explain in detail.  Let $0_m$ and $1_m$ denote the $m \times m$ zero and identity matrix respectively and define the matrix $J$ by
\[
J:= \begin{bmatrix} 0_m & 1_m \\ -1_m & 0_m\end{bmatrix}.
\]
If we denote $x^T$ to be the transpose of a matrix $x$, we define the symplectic group as follows, 
\[
\operatorname{Sp}_{2m}(\bb{F}_2):=\{x \in \operatorname{GL}_{2m}(\bb{F}_2) : x^T J x = J\}.
\]
We can also view $\operatorname{Sp}_{2m}(\bb{F}_2)$ as the set of linear transformations which preserve the following alternating bilinear form $\psi$.
\begin{align*}
\psi: \bb{F}_2^{2m}\times \bb{F}_2^{2m}&\rightarrow \bb{F}_2 \\ 
\Bigg(\begin{bmatrix} u_1 \\ u_2 \end{bmatrix}, \begin{bmatrix} v_1 \\ v_2 \end{bmatrix}\Bigg) &\mapsto u_1 \cdot v_2 - u_2 \cdot v_1
\end{align*}
That is to say, 
\[
\operatorname{Sp}_{2m}(\bb{F}_2) := \{x \in \operatorname{GL}_{2m}(\bb{F}_2) : \psi(xu,xv) = \psi(u,v) \quad \forall u,v \in \bb{F}_2^{2m}\}.
\]

However, $\psi$ is also a symmetric form over $\bb{F}_2$ and the doubly transitive action we are interested in involves orthogonal groups of associated quadratic forms.
 
\begin{defn}
If $V$ is a vector space over $\bb{F}_2$, then a quadratic form on $V$ is a function $q:V \rightarrow \bb{F}_2$ such that 
\begin{enumerate}
\item $q(\lambda v) = \lambda^2 q(v)$ for all $\lambda \in \bb{F}_2, v \in V$;
\item The function $f(u,v):= q(u+v)-q(u)-q(v)$ is a symmetric bilinear form.
\end{enumerate}
\end{defn}
Over $\bb{F}_2$, a quadratic form determines a symmetric bilinear form, but many quadratic forms can determine the same symmetric bilinear form.
%If $V$ is a $2m$ dimensional vector space over $\bb{F}_2$, then there are two quadratic forms on $V$ up to isometry, $Q^+:V \rightarrow \bb{F}_2$ and $Q^-:V \rightarrow\bb{F}_2$.  To distinguish these quadratic forms,  we look at isotropic vectors. A vector $v \in V$ is isotropic for a quadratic form $Q$ on $V$ if $Q(v) = 0$.  The Witt index of a quadratic form is the largest dimension of an isotropic subspace (that is, a subspace consisting of only isotropic vectors).  The quadratic form $Q^+$ has Witt index $m$ while the quadratic form $Q^-$ has Witt index $m-1$.  
Up to isometry, there are two quadratic forms on $\bb{F}_2^{2m}$, given by
\begin{align*}
Q^+(v_1,\ldots,v_{2m}) &=v_{1}v_{m+1} + \cdots + v_{m}v_{2m},\\
Q^-(v_1,\ldots,v_{2m})&= v_{1}v_{m+1}+ \cdots + v_{m}v_{2m}+v_m+v_{2m}.
\end{align*}

We define two orthogonal groups, $GO_{2m}^+(\bb{F}_2)$ and $GO_{2m}^-(\bb{F}_2)$, as the isometry groups of these quadratic forms,
\[
\operatorname{GO}^{\pm}_{2m}(\bb{F}_2) = \{x \in \operatorname{GL}_{2m}(\bb{F}_2) : Q^{\pm}(xv) = Q^{\pm}(v) \quad \forall v \in \bb{F}_2^{2m}\}.
\]

Both $Q^+$ and $Q^-$ determine the same bilinear symmetric form, $\psi$, and thus we have that $\operatorname{GO}^+_{2m}(\bb{F}_2) \subseteq \operatorname{Sp}_{2m}(\bb{F}_2)$ and $\operatorname{GO}^-_{2m}(\bb{F}_2) \subseteq \operatorname{Sp}_{2m}(\bb{F}_2)$.  Furthermore, $\operatorname{Sp}_{2m}(\bb{F}_2)$ acts doubly transitively on the cosets by either of the aforementioned subgroups.

Applying Theorem 3.1 to this situation shows that if $F$ is a totally real number field whose Galois closure $F^c$ has Galois group $\operatorname{Sp}_{2m}(\bb{F}_2)$ and $F$ is the fixed field by either $\operatorname{GO}^+_{2m}(\bb{F}_2)$ or $\operatorname{GO}^-_{2m}(\bb{F}_2)$, then the Colmez conjecture holds for $E:=kF$.

Another class of examples is the unitary groups $\operatorname{PSU}_3(\bb{F}_q)$ and $\operatorname{PGU}_3(\bb{F}_q)$ where $q$ is a prime power.  To describe these groups and the relevant subgroups, we will follow the convention of \cite{Diximer}.  Let $\varphi$ be the following bilinear form on $\bb{F}_{q^2}^3$.
\begin{align*}
\varphi: \bb{F}_{q^2}^3 \times \bb{F}_{q^2}^3 &\rightarrow \bb{F}_{q^2} \\
\Bigg(\begin{bmatrix}
u_1\\ u_2 \\u_3\end{bmatrix} ,\begin{bmatrix}v_1 \\ v_2 \\v_3 \end{bmatrix}\Bigg) &\mapsto u_1v_3^q + u_2v_2^q+u_3v_1^q
\end{align*}

We define $\operatorname{GU}_3(\bb{F}_q)$ to be the matrices which preserve this form,
\[
\operatorname{GU}_3(\bb{F}_q):= \{A \in \operatorname{GL}_3(\bb{F}_{q^2}) : \varphi(Au,Av) = \varphi(u,v)\quad \forall u,v \in \bb{F}_{q^2}^3\}.
\]
The action of $\operatorname{GU}_3(\bb{F}_q)$ on the 1-dimensional subspaces of $\bb{F}_{q^2}^3$ defines the projective group $\operatorname{PGU}_3(\bb{F}_q)$ and taking those matrices of determinant 1 defines the group $\operatorname{PSU}_3(\bb{F}_q)$.  

An isotropic vector $v \in \bb{F}_{q^2}^3$ is a vector such that $\varphi(v,v) = 0$.  The groups $\operatorname{PGU}_3(\bb{F}_q)$ and $\operatorname{PSU}_3(\bb{F}_q)$ both act doubly transitively on the set of 1-dimensional isotropic subspaces of $\bb{F}_{q^2}^3$, so therefore we may take $H$ to be the stabilizer of any 1-dimensional isotropic subspace. A quick calculation shows that the stabilizer in $\operatorname{PSU}_3(\bb{F}_q)$ and $\operatorname{PGU}_3(\bb{F}_q)$ of the subspace spanned by $\begin{bmatrix} 1 \\ 0 \\ 0\end{bmatrix}$ is exactly the upper triangular matrices of the respective group

%To get upper triangular, look at \varphi(e_1,e_2) = \varphi(Ae_1,Ae_2)

Applying Theorem 3.1 to this group action shows that the Colmez conjecture holds for $E:=kF$ where $k$ is an imaginary quadratic field, $F$ is a totally real number field such that $\operatorname{Gal}(F^c/\bb{Q})\cong \operatorname{PSU}_3(\bb{F}_q)$ or $\operatorname{PGU}_3(\bb{F}_q)$ and $\operatorname{Gal}(F^c/F)$ is the subgroup of upper triangular matrices. 

There are a two more infinite families and a few more sporadic examples which we list here. 
\begin{itemize}
\item The Suzuki groups $\operatorname{Sz}(q)$ for $q$ an odd power of 2 is a doubly transitive subgroup of $S_{q^2+1}$.
\item The Ree groups $\operatorname{R}(q)$ for $q$ an odd power of 3 is a doubly transitive subgroup of $S_{q^3+1}$.
\item The Mathieu groups $M_{11}, M_{12}, M_{22},M_{23},M_{24}$ are doubly transitive subgroups of $S_{11}, S_{12}, S_{22}, S_{23}$, and $S_{24}$ respectively.  
\item The Mathieu group $M_{11}$ is a doubly transitive subgroup of $S_{12}$ and $\operatorname{PSL}_2(\bb{F}_{11})$ is a transitive subgroup of $S_{11}$.  
\item The alternating group $A_7$ is a transitive subgroup of $S_8$.  
\item The Higman-Sims group $\operatorname{HS}$ is a doubly transitive subgroup of $S_{176}$.
\item The Conway group $\operatorname{Co}_3$ is a doubly transitive subgroup of $S_{276}$.

\end{itemize}

%We should also make a comment regarding recent work of Yuan and Zhang \cite{YZ}.  In their work, they proved that the Colmez conjecture holds on average.  That is, for $E$ a CM field (not necessarily unitary), they show that 
%\[
%\sum_{\Phi} \height(A_{\Phi}^0) = \sum_{\Phi} Z(0,A_{\Phi}^0)
%\]
%and furthermore, they determine the value of this expression.  To do this, they calculate a sort of average of the height of two nearby CM types, where $\Phi_1$ and $\Phi_2$ are nearby CM types if $\Phi_1 \cap \Phi_2 = n-1$.  Since any CM type of signature $(n-2,2)$ is equivalent to $\Phi_{\{1,i\}}$ for some $i \in \{2,\ldots,n\}$, every pair of CM types of signature $(n-2,2)$ constitutes a nearby pair of CM types.

\section{Acknowledgements} 
The author would like to thank Tonghai Yang for his assistance throughout work on this paper. The author would also like to thank Mois\'{e}s Herrad\'{o}n Cueto for his help with doubly transitive groups. This work was done with the partial support of National Science Foundation grant DMS-1502553.

\bibliographystyle{alpha}
\bibliography{/Users/salvatoreparenti/Desktop/TexFiles/mybib}  

\begin{thebibliography}{{LMF}17}

\bibitem[BHR13]{BHRD}
J.~{Bray}, D.~{Holt}, and C.~{Roney-Dougal}.
\newblock {\em The Maximal Subgroups of the Low-Dimensional Finite Classical
  Groups}.
\newblock London Mathematical Society, 2013.

\bibitem[{Col}93]{Col}
P.~{Colmez}.
\newblock P{\'e}riodes des vari{\'e}t{\'e}s ab{\'e}liennes {\`a} multiplication
  complexe.
\newblock {\em Ann. of Math.}, 138:625--683, 1993.

\bibitem[{Col}98]{ColSurLa}
P.~{Colmez}.
\newblock Sur la hauteur de faltings des vari{\'e}t{\'e}s {\`a} multiplication
  complexe.
\newblock {\em Compositio Math.}, 111:359--368, 1998.

\bibitem[DM96]{Diximer}
J.~{Dixon} and B.~{Mortimer}.
\newblock {\em Permutation Groups}.
\newblock Graduate Texts in Mathematics, 1996.

\bibitem[{LMF}17]{lmfdb}
The {LMFDB Collaboration}.
\newblock The {L}-functions and modular forms database.
\newblock \url{http://www.lmfdb.org}, 2017.
\newblock [Online; accessed December 2017].

\bibitem[{Par}17]{ParentiPSL2}
S.~{Parenti}.
\newblock Unitary {$PSL_2$} {CM} fields and the {C}olmez conjecture.
\newblock {\em ArXiv e-prints}, 2017.

\bibitem[{Tsi}18]{TsimermanAndreOort}
J.~{Tsimerman}.
\newblock The {Andr{\'e}-Oort} conjecture for ${A_g}$.
\newblock {\em Ann. of Math.}, 187(2):379--390, 2018.

\bibitem[YY17]{YangYin}
T.~{Yang} and H.~{Yin}.
\newblock {CM} fields of dihedral type and the colmez conjecture.
\newblock {\em Manuscripta Math.}, 2017.

\end{thebibliography}

\end{document}